\documentclass[12pt, a4paper]{amsart}   	% use "amsart" instead of "article" for AMSLaTeX format
\usepackage[top=1.15in, bottom=1.15in, left=1.17in, right=1.17in]{geometry}

\usepackage[colorlinks=true, pdfstartview=FitV, linkcolor=blue, citecolor=blue, urlcolor=blue, breaklinks=true]{hyperref}
\usepackage{amssymb,amscd,graphics, amsfonts, multicol}
\usepackage{graphics}
\usepackage{stmaryrd}
\usepackage{DotArrow}
\usepackage{amsmath, amsthm,wasysym}
\usepackage{epsf}
\usepackage{xypic}
\numberwithin{equation}{section}
\usepackage{tikz-cd}
\usepackage{color}
\usepackage[normalem]{ulem}
\theoremstyle{plain}

\newtheorem{theorem}{Theorem}[section]
\newtheorem{corollary}[theorem]{Corollary}
\newtheorem{proposition}[theorem]{Proposition}
\newtheorem{lemma}[theorem]{Lemma}

\theoremstyle{definition}

\newtheorem{remark}[theorem]{Remark}

\newtheorem{example}[theorem]{Example}
\newtheorem{definition}[theorem]{Definition}
\definecolor{fondo}{rgb}{0.898,0.996,0.898}

\newcommand{\R}{{\mathbb R}}

\newcommand{\Z}{{\mathbb Z}}
\newcommand{\N}{{\mathbb N}}

\newcommand{\mcA}{{\mathcal A}}

\newcommand{\mcI}{{\mathcal I}}

\DeclareMathOperator{\GL}{GL}
\DeclareMathOperator{\SL}{SL}

\DeclareMathOperator{\LS}{\Pi}

\DeclareMathOperator{\LH}{P}
\DeclareMathOperator{\initial}{in}

\DeclareMathOperator{\supp}{supp}
\DeclareMathOperator{\Spec}{Spec}

\newcommand{\CC}{\mathbb{C}}
\newcommand{\ZZ}{\mathbb{Z}}
\newcommand{\NN}{\mathbb{N}}

\newcommand{\RR}{\mathbb{R}}
\title[Lecture hall polytopes and Lakshmibai-Seshadri paths]{Lecture hall polytopes and\\Lakshmibai-Seshadri paths}

\author{Rocco Chiriv\`i}
\address{Dipartimento di Matematica e Fisica ``Ennio De Giorgi'', Universit\`a del Salento, Lecce, Italy}
\email{rocco.chirivi@unisalento.it}

\author{Martina Costa Cesari}
\address{European Research Council Executive Agency, Brussels, Belgium}
\email{martina.costa.cesari@live.it}

\author{Xin Fang}
\address{Lehrstuhl f\"ur Algebra und Darstellungstheorie, RWTH Aachen, Pontdriesch 10-16, 52062 Aachen, Germany}
\email{xinfang.math@gmail.com}

\author{Peter Littelmann}
\address{Department Mathematik/Informatik, Universit\"at zu K\"oln, 50931, Cologne, Germany}
\email{peter.littelmann@math.uni-koeln.de}

\begin{document}

\begin{abstract} Using a bijection between the lattice points in a lecture hall polytope and Lakshimibai-Seshadri (L-S) paths, we prove the Koszul property of lecture hall polytopes in complete generality, and give new proofs for their Integral Decomposition Property and for a criterion on Gorenstein property.
\end{abstract}

\maketitle

%%%%%%%%%%%%%%%%%%%%%%%%%%%%%%%%%%%%%%%%%%%%%%%%%%%%%%%%%%%%%%%%%%
%%%%%%%%%%%%%%%%%%%%%%%%%%%%%%%%%%%%%%%%%%%%%%%%%%%%%%%%%%%%%%%%%%
%%%%%%%%%%%%%%%%%%%%%%%%%%%%%%%%%%%%%%%%%%%%%%%%%%%%%%%%%%%%%%%%%%
%%%%%%%%%%%%%%%%%%%%%%%%%%%%%%%%%%%%%%%%%%%%%%%%%%%%%%%%%%%%%%%%%%
%%%%%%%%%%%%%%%%%%%%%%%%%%%%%%%%%%%%%%%%%%%%%%%%%%%%%%%%%%%%%%%%%%

\section{Introduction}

The original lecture hall polytopes have been introduced by Bousquet-Mélou and Eriksson \cite{BousquetEriksson} in 1997, and they have been generalized by Savage and Schuster \cite{SS} in 2012 to $b$--lecture hall polytopes, where $b = (b_1, \ldots, b_n)$ is a finite sequence of positive numbers. The lecture hall polytopes introduced in \cite{BousquetEriksson} are $b$--lecture hall polytopes for $b=(1,2,\ldots,n)$. In the literature related to lecture hall polytopes the parameter is usually denoted by $s = (s_1, \ldots,s_n)$; here we prefer to denote it by $b$ since in the representation theory they are known as \emph{bonds} and are interpreted as a kind of vanishing orders \cite{CFL1,CFL3,CFL4}.

In \cite{SS} the authors study the $b$--lecture hall partitions from polyhedral geometric point of view. From then to recent years, the geometric and combinatorial properties of these polytopes have been extensively studied in combinatorics. For example, their Integral Decomposition Property (IDP) was proved in \cite{HibiOlsenTsuchiya, BrandenSolus}, and a criterion on the Gorenstein property has been studied still in \cite{HibiOlsenTsuchiya}.

On the other hand the Lakshmibai-Seshadri paths (L-S paths) \cite{LittelmannPaths2, LittelmannPaths} are key combinatorial objects from representation theory parametrizing the basis of the irreducible representations of a complex semi-simple Lie algebra (see also \cite{LittelmannPaths3,CFL0,CFL1,CFL3,CFL4}). The L-S paths can be used to construct a standard monomial theory; this is studied, for example, in \cite{Chirivi, LLM}.

The main new idea of this paper is to show that the integral points of the $s$-lecture hall polytopes (a.k.a. the $b$--lecture hall partitions) are in bijection with L-S paths with bonds $b$ via a simple linear transformation. This bridge links the lecture hall polytopes to representation theory.

Such a bijection allows us to translate results on L-S paths \cite{Chirivi} into those on $b$--lecture hall polytopes: the IDP is a direct consequence of the canonical decomposition of an L-S path (Section \ref{Sec:3}), and the criterion on Gorenstein property follows from a finite group action on the L-S paths (Section \ref{Sec:5}).

Another application is the following new result: we prove the Koszul property of arbitrary $b$--lecture hall polytopes. Such a property has been already proved for $b$--lecture hall polytopes with special $b$ in \cite{BrandenSolus} by working out a square-free quadratic Gr\"obner basis. Our proof of the Koszul property is established by an explicit description of a quadratic straightening law on the semi-group algebra of the L-S paths \cite{LLM, Chirivi} (see also \cite{CFL3} for a generalization), which is isomorphic to the Ehrhart ring of the lecture hall polytope. This straightening law allows us to construct a quadratic Gr\"obner basis, hence to prove the Koszul property.

\smallskip

In this paper we consider only L-S paths for a finite totally ordered set with bonds but the general context of a finite partially ordered set. Since the proofs are usually simpler for a totally ordered set, we have decided to include the full proofs in this paper to make it self contained and, in our opinion, more readable.

Finally, the bijection we propose here between L-S paths and integral points in lecture hall polytopes is not just at combinatorial level. We call \emph{lecture hall toric varieties} (for short LH-toric varieties) the projective toric varieties whose homogeneous coordinate rings are the Ehrhart algebras of $b$--lecture hall polytopes. By \cite{Chirivi}, the lecture hall toric varieties are the irreducible components of a semi-toric degeneration of Schubert varieties of semi-simple algebraic groups. Moreover each lecture hall toric variety appears as such a component as we show in the last section of this article.

\smallskip

\noindent{\bf Organization of the paper:} The paper is based on the bijection between lattice points in the lecture hall polytope and L-S paths, which is established in Section \ref{Sec:2}. The IDP is proved in Section \ref{Sec:3} and a criterion of the Gorenstein property is given in \ref{Sec:5}. The isomorphism between the semi-group algebra of the L-S paths and the Ehrhart algebra of the lecture hall polytope is explained in Section \ref{section:semigroup-algebra}. The Koszul property is shown in Section \ref{Sec:6}. In Section \ref{Sec:7} we show that there exists a semi-toric degeneration of certain Schubert varieties to the union of toric varieties associated to the $b$--lecture hall polytopes.

\vskip 10pt

\noindent{\bf Acknowledgements:} The work of X.F. is funded by the Deutsche Forschungsgemeinschaft: “Symbolic Tools in Mathematics and their Application” (TRR 195, project-ID 286237555). 

\smallskip
\noindent{\bf Disclaimer:} The views expressed are purely those of the authors and may not in any circumstances be regarded as stating an official position of the European Research Council Executive Agency and the European Commission.
\smallskip

\noindent{\bf Tool and computational resource disclosure:} During the preparation of this manuscript various AI tools have been used for bibliographical research only.

%%%%%%%%%%%%%%%%%%%%%%%%%%%%%%%%%%%%%%%%%%%%%%%%%%%%%%%%%%%%%%%%%%
%%%%%%%%%%%%%%%%%%%%%%%%%%%%%%%%%%%%%%%%%%%%%%%%%%%%%%%%%%%%%%%%%%
%%%%%%%%%%%%%%%%%%%%%%%%%%%%%%%%%%%%%%%%%%%%%%%%%%%%%%%%%%%%%%%%%%
%%%%%%%%%%%%%%%%%%%%%%%%%%%%%%%%%%%%%%%%%%%%%%%%%%%%%%%%%%%%%%%%%%
%%%%%%%%%%%%%%%%%%%%%%%%%%%%%%%%%%%%%%%%%%%%%%%%%%%%%%%%%%%%%%%%%%

\section{The bijection}\label{Sec:2}

We begin by defining the L-S paths for a set of bonds. Fix an integer $n\ge 1$ and a sequence of positive integers
\[
b = (b_1, b_2, \dots, b_n), \qquad b_j \in \ZZ_{>0},
\]
called the \emph{bonds}. The objects we are going to define all depend on the choice of the bonds $b$. For convenience we set $b_0 = b_{n+1} := 1$.
 
Let $L(b)$ be the lattice generated by the vectors $v_1,\dots,v_n, v_{n+1} \in \RR^{n+1}$ with
\[
v_j := \frac{e_j - e_{j+1}}{b_j}, \, j = 1,\ldots,n,\qquad\textrm{and}\qquad v_{n+1} := e_{n+1}
\]
where $e_1,\dots,e_n,e_{n+1}$ is the canonical basis of $\RR^{n+1}$. It is clear that $L(b)$ contains $\ZZ^{n+1}$.

\begin{definition}
The set of \emph{L-S paths}  with bond $b$ is
\[
\Pi(b) := L(b) \cap \RR^{n+1}_{\ge 0}.
\]
For $\pi = (a_1,\dots,a_{n+1})\in \Pi(b)$ define the \emph{support} $\supp\pi$ as the set of indices $1\leq j\leq n+1$ such that $a_j\neq 0$ and set $\deg(\pi) := a_1 + \cdots + a_{n+1} $, the \emph{degree} of the L-S path $\pi$. Note that if $\pi = m_1v_1 + \cdots + m_{n+1}v_{n+1}$ then $\deg(\pi) = m_{n+1}$; in particular $\deg(\pi)$ is a non-negative integer. We denote by $\Pi(b)_k$ the set of L-S paths of degree $k$.
\end{definition}

\begin{remark}
Thus $\Pi(b)$ is a submonoid of $L(b)$ and $\deg$ restricts to a monoid homomorphism $\Pi(b)\to\ZZ_{\ge 0}$.
\end{remark}

We give a different characterization of the L-S paths.
\begin{proposition}\label{prop:LS-characterization}
    A sequence of non-negative real numbers $\pi = (a_1,\dots,a_{n+1})$ is an L-S path if and only if
    \[
    \left\{
    \begin{array}{rcl}
        b_1 a_1 & \in &  \NN,\\
        b_2(a_1 + a_2) & \in &  \NN,\\
        & \vdots\\
        b_n(a_1 + \cdots + a_n) & \in &  \NN,\\
        a_1 + \cdots + a_{n+1} & \in &  \NN.
    \end{array}
    \right.
    \]
\end{proposition}
\begin{proof}
Let $\pi=(a_1,\dots,a_{n+1})\in\RR^{n+1}_{\ge0}$ and write $\pi = \sum_{j=1}^{n+1} m_j v_j$ with $m_j\in\RR$. A straightforward computation shows that $m_j = b_j(a_1 + \cdots + a_j)$ for $j = 1, \ldots, n+1$. Since $\pi\in L$ if and only if $m_j\in\ZZ$ for all $j$, the result follows.
\end{proof}

\smallskip

Let $N$ be the least common multiple of the bonds $b_1,\dots,b_n$ and let $N_j$ be the least common multiple of $b_j$ and $b_{j-1}$ for $j=1,\ldots,n+1$.
\begin{lemma}\label{lem:coordinates-in-frac}
    If $\pi = (a_1,\dots,a_{n+1})$ is an L-S path then $a_j\in \frac{1}{N_j}\ZZ$ for $j=1,\ldots,n+1$.
\end{lemma}
\begin{proof}
By Proposition~\ref{prop:LS-characterization}, $a_1 + \cdots + a_j\in\frac{1}{b_j}\ZZ$ and $a_1 + \cdots + a_{j+1}\in\frac{1}{b_{j+1}}\ZZ$. Subtracting these two equations we get $a_j\in\frac{1}{N_j}\ZZ$ as claimed.
\end{proof}

\smallskip

The \emph{Lecture Hall Polytope} with parameters $ b = (b_1,\ldots,b_n)$ is defined as
\[
\LH(b) := \left.\left\{
(x_1, x_2, \ldots, x_n)\,\in \mathbb{R}^n
\,\right|\,
0\leq\frac{x_1}{b_1}\leq\ldots\leq \frac{x_n}{b_n}\leq 1\right\}.
\]
It is clear that the dilation $k\cdot\LH(b)$ is given by the same inequalities with the rightmost $1$ replaced by $k$.

Consider the linear map from $\R^{n+1}$ to $\R^n$ defined by
\[
\varphi \,:\, (a_1, \ldots, a_n, a_{n+1})
\,\longmapsto\, (x_1, \ldots, x_n)
\]
where $x_i = b_i(a_1 + \cdots + a_i)$ for all $i=1,\ldots,n$.

\begin{proposition}\label{prop_bijection-LS-LH}
    The restriction of the map $\varphi$ is a bijection
    \[
    \varphi:\Pi(b)_k \longrightarrow k\cdot\LH(b)\,\cap\,\Z^n
    \]
    from the set of L-S paths of degree $k$ to the set of integral points in the $k$--dilated lecture hall polytope.
\end{proposition}
\begin{proof} Let $\pi = (a_1,\ldots,a_{n+1})$ be an L-S path of degree $k$. Then $x_i := b_i(a_1 + \cdots + a_i)\in\ZZ$ by Proposition \ref{prop:LS-characterization}. Hence $\varphi(\pi) = (x_1, \ldots, x_{n+1})\,\in\,\Z^n$.

Moreover $x_i /b_i = a_1 + \cdots + a_i$, so using $\pi\in\RR^{n+1}_{\geq 0}$ we find
\[
\displaystyle 0\leq \frac{x_1}{b_1} \leq \frac{x_2}{b_2} \leq \cdots \leq \frac{x_n}{b_n}.
\]
Finally $x_n / b_n = a_1 + \cdots + a_n \leq a_1 + \cdots + a_{n+1}= k$. We have shown that $\varphi(\pi)\in k\cdot P(b) \cap \ZZ^n$.

Conversely, let $(x_1, \ldots, x_n)\in k\cdot\LH(b)\cap\Z^n$ and set $x_0 := 0$. Define $a_i := x_i / b_i - x_{i-1} / b_{i-1}$, for $i = 1,\ldots, n$, and $a_{n+1} := k - x_n / b_n$. It is clear that $\pi := (a_1, \ldots, a_{n+1})\in \RR^{n+1}_{\geq 0}$. Since $b_i(a_1 + \cdots + a_i) = x_i\in\ZZ$, we find that $\pi$ is an L-S path of degree $k$ and, finally, $\varphi(\pi) = (x_1, \ldots, x_n)$.
\end{proof}
\begin{remark}\rm
The first connection between LS-paths, LH-polytopes (the name came up 15 years later) and lattices has been observed by R. Dehy,
see \cite[Section 4.2.2]{Dehy}. She studies conditions under which these polytopes can be glued together to get a polytopal character model for Demazure modules.   
\end{remark}

%%%%%%%%%%%%%%%%%%%%%%%%%%%%%%%%%%%%%%%%%%%%%%%%%%%%%%%%%%%%%%%%%%
%%%%%%%%%%%%%%%%%%%%%%%%%%%%%%%%%%%%%%%%%%%%%%%%%%%%%%%%%%%%%%%%%%
%%%%%%%%%%%%%%%%%%%%%%%%%%%%%%%%%%%%%%%%%%%%%%%%%%%%%%%%%%%%%%%%%%
%%%%%%%%%%%%%%%%%%%%%%%%%%%%%%%%%%%%%%%%%%%%%%%%%%%%%%%%%%%%%%%%%%
%%%%%%%%%%%%%%%%%%%%%%%%%%%%%%%%%%%%%%%%%%%%%%%%%%%%%%%%%%%%%%%%%%

\section{The Integral Decomposition Property}\label{Sec:3}

The Integral Decomposition Property has already been proved in \cite{HibiOlsenTsuchiya} for certain classes of lecture hall polytopes, and in \cite{BrandenSolus} for all lecture hall polytopes (and certain generalizations).

This property is a direct consequence of the existence of the canonical decomposition of L-S paths proved in \cite{Chirivi}. We report here essentially the same proof of that paper.

\begin{proposition}\label{prop:ls-normality}
\begin{itemize}
\item[(i)] The monoid $\Pi(b)$ of L-S paths is saturated in $L(b)$, i.e. if $\pi\in L(b)$ and $k\pi\in\LS(b)$ for some non-negative integer $k$ then $\pi\in\LS(b)$.
\item[(ii)] If $\pi,\eta\in\LS(b)$ and $\pi - \eta\in\RR^{n+1}_{\ge0}$ then $\pi - \eta\in\LS(b)$.
\item[(iii)] If $\pi\in \LS(b)_k$ then there exists a unique decomposition $\pi = \pi_1 + \cdots + \pi_k$, with $\pi_i\in\LS(b)_1$ for $i=1,\ldots,k$, such that $\max\supp\pi_j\leq \min\supp\pi_{j+1}$ for $j=1,\ldots,k-1$. In particular, $\LS(b)$ is a normal monoid.
\end{itemize} 
\end{proposition}
\begin{proof}
If $\pi\in L(b)$ and $k\pi\in\LS(b)$ then $k\pi\in\RR^{n+1}_{\ge0}$, so $\pi\in\RR^{n+1}_{\ge0}$. But then we have $\pi\in L(b)\cap \RR^{n+1}_{\ge0} = \LS(b)$. This proves \textrm{(i)}.

For \textrm{(ii)}, if $\pi,\eta\in\LS(b)$ and $\pi - \eta\in\RR^{n+1}_{\ge0}$ then $\pi - \eta\in L(b)$ since $L(b)$ is a lattice, so $\pi - \eta\in L(b)\cap \RR^{n+1}_{\ge0} = \LS(b)$.

In order to prove \textrm{(iii)} we proceed by induction on $k$. Let $\pi = (a_1, \ldots, a_{n+1})\in\LS(b)_k$. If $k = 0$, there is nothing to prove. Assume $k \geq 1$ and let $j$ be minimal such that $a_1 + \ldots + a_j \geq 1$. Then we can write $a_1 + \ldots + a_j = 1 + c$ with $c\geq 0$. Define $\pi_1 = (a'_1, \ldots, a'_{n+1})$ by setting $a'_i = a_i$ for all $i < j$, $a'_j = a_j - c$, and $a'_i = 0$ for all $i > j$. Then, it is clear that $\pi_1\in\LS(b)_1$ by the characterization in Proposition \ref{prop:LS-characterization}. Moreover $\pi - \pi_1$ has non-negative coordinates by construction and by \textrm{(ii)} it is an LS path of degree $k-1$. By induction, we can write $\pi - \pi_1 = \pi_2 + \ldots + \pi_k$ with $\pi_2, \ldots, \pi_k\in\LS_1(b)$ the unique L-S paths of degree $1$ satisfying the required conditions. This completes the proof.
\end{proof}
We call the decomposition in \textrm{(iii)} the \emph{canonical decomposition} of $\pi$ and we write $\pi = \pi_1 \oplus \cdots \oplus \pi_k$.

\begin{theorem}\label{theorem_IDP_LH}
    Let $b=(b_r,\ldots,b_1)$ be any set of bonds. The lecture hall polytope $P(b)$ has the Integral Decomposition Property: for any $k\geq 0$ and any $p = (x_1, \ldots, x_n)\in k\cdot P(b)\,\cap\,\Z^n$, there exist $p_1, \ldots, p_k\in P(b)\cap\Z^n$ such that $p = p_1 + \ldots + p_k$.
\end{theorem}
\begin{proof} By Proposition~\ref{prop_bijection-LS-LH}, there exists $\pi\in\LS(b)_k$ such that $\varphi(\pi) = p$. By Proposition~\ref{prop:ls-normality}, we can write $\pi = \pi_1 + \ldots + \pi_k$ as a sum of $k$ L-S paths of degree $1$. Again by Proposition~\ref{prop_bijection-LS-LH}, $p_j := \varphi(\pi_j)$ is a lattice point in $P(b)$ for all $j=1,\ldots,k$. Finally, since $\varphi$ is linear, we have $p = \varphi(\pi) = \varphi(\pi_1) + \ldots + \varphi(\pi_k) = p_1 + \ldots + p_k$.
\end{proof}

%%%%%%%%%%%%%%%%%%%%%%%%%%%%%%%%%%%%%%%%%%%%%%%%%%%%%%%%%%%%%%%%%%
%%%%%%%%%%%%%%%%%%%%%%%%%%%%%%%%%%%%%%%%%%%%%%%%%%%%%%%%%%%%%%%%%%
%%%%%%%%%%%%%%%%%%%%%%%%%%%%%%%%%%%%%%%%%%%%%%%%%%%%%%%%%%%%%%%%%%
%%%%%%%%%%%%%%%%%%%%%%%%%%%%%%%%%%%%%%%%%%%%%%%%%%%%%%%%%%%%%%%%%%
%%%%%%%%%%%%%%%%%%%%%%%%%%%%%%%%%%%%%%%%%%%%%%%%%%%%%%%%%%%%%%%%%%

\section{The semigroup algebra and the Ehrhart ring}\label{section:semigroup-algebra}

Now we define a semigroup algebra associated to the L-S paths. Recall that we have defined $N_k$, $k = 1, \ldots, n+1$ as the least common multiple of $b_k$ and $b_{k-1}$. Let us denote by $R := \CC[t_1^{1/N_1},\dots,t_{n+1}^{1/N_{n+1}}]$ the polynomial algebra in the indeterminates $t_1^{1/N_1},\dots,t_{n+1}^{1/N_{n+1}}$. Given an L-S path $\pi = (a_1,\dots,a_{n+1})$, the monomial $t^\pi := t_1^{a_1}\cdots t_{n+1}^{a_{n+1}}$ is a well-defined monomial of $R$ by Lemma \ref{lem:coordinates-in-frac}.

\begin{definition}\label{def:semigroup-algebra}
The \emph{semigroup algebra} associated to the L-S paths is
\[
A(b) := \CC[\LS(b)] = \bigoplus_{\pi\in \LS(b)} \CC\,\cdot t^\pi \ \subseteq\ R
\]
whose multiplication is given by bilinearly extending $t^\pi\cdot t^\eta := t^{\pi+\eta}$, $\pi,\eta\in\LS(b)$.
\end{definition}
\begin{proposition}\label{prop:A-algebra}
The algebra $A(b)$ is graded by the degree of L-S paths  
\[
\ A(b) = \bigoplus_{k\geq 0} A(b)_k, \quad A(b)_k := \langle t^\pi \mid \pi\in\LS(b)_k\rangle_\CC.
\]
Morever $A(b)$ is finitely generated by the elements $t^\pi$, $\pi\in\LS(b)_1$, of degree $1$, and it is a normal subalgebra of $R$.
\end{proposition}
\begin{proof}
By Proposition \ref{prop:ls-normality}, $A$ is generated by $t^\pi$, $\pi\in\LS(b)_1$ (a finite set), and it is normal.
\end{proof}

Let $T := \Spec\CC[L(b)]$, an $(n+1)$-dimensional algebraic torus with character lattice $X^\ast(T)$ canonically identified with $L(b)$. When we want to consider an element $\pi$ of $L(b)$ as a character of $T$, we write $\chi^\pi$. Since $\LS(b)$ is a submonoid of $L(b)$, $T$ acts naturally on $A(b)$. The following proposition is clear.

\begin{proposition}\label{prop:T-eigenspaces}
As a $T$--module, $A(b)$ decomposes as the direct sum of one-dimensional eigenspaces $\CC\,\cdot\,t^\pi$, $\pi\in\LS(b)$, of weight $\chi^\pi$.
\end{proposition}

\smallskip

Now we see the natural object on the lecture hall polytope side corresponding to $A(b)$. The Ehrhart ring of $P(b)$ is defined as follows. Given $p=(p_1,\ldots,p_n)\in kP(b)\cap \Z^n$ with $k\in\N$, let $t^{(p,k)}:=t_1^{p_1}t_2^{p_2}\cdots t_n^{p_n}t_{n+1}^k$. The Ehrhart ring of $P(b)$ is the graded algebra
\[
E(b) := \bigoplus_{k\geq 0}E(b)_k, \quad E(b)_k := \langle t^{(p,\,k)} \mid p\in kP(b)\cap\Z^n\rangle_\CC.
\]
with multiplication defined by $t^{(p,\,k)}\cdot t^{(q,\,h)} = t^{(p+q,\, k+h)}$.

\begin{proposition}\label{prop:LH-Ehrhart-ring}
    The linear extension of
    \[
    A(b)\ni t^\pi\longmapsto t^{(\varphi(\pi),\,\deg\pi)}\in E(b)
    \]
    is an isomorphism of graded algebras.
\end{proposition}
\begin{proof}
Since $\varphi$ is a bijection from $\LS(b)_k$ to $kP(b)\cap\Z^n$ for all $k\geq 0$ by Proposition \ref{prop_bijection-LS-LH}, the map is well defined and is a bijection between two vector space basis. Hence its linear extension is a vector space isomorphism. It is also an algebra homomorphism since $\varphi$ is linear and $\deg(\pi + \pi') = \deg\pi + \deg\pi'$ for all $\pi, \pi'\in\LS(b)$.
\end{proof}

%%%%%%%%%%%%%%%%%%%%%%%%%%%%%%%%%%%%%%%%%%%%%%%%%%%%%%%%%%%%%%%%%%
%%%%%%%%%%%%%%%%%%%%%%%%%%%%%%%%%%%%%%%%%%%%%%%%%%%%%%%%%%%%%%%%%%
%%%%%%%%%%%%%%%%%%%%%%%%%%%%%%%%%%%%%%%%%%%%%%%%%%%%%%%%%%%%%%%%%%
%%%%%%%%%%%%%%%%%%%%%%%%%%%%%%%%%%%%%%%%%%%%%%%%%%%%%%%%%%%%%%%%%%
%%%%%%%%%%%%%%%%%%%%%%%%%%%%%%%%%%%%%%%%%%%%%%%%%%%%%%%%%%%%%%%%%%

\section{The Gorenstein property}\label{Sec:5}

We can now give a new Gorenstein criterion for a lecture hall polytope using our L-S path setting. This property has been studied initially in \cite{HibiOlsenTsuchiya} and then in \cite{BrandenSolus} where a general criterion for gorensteiness is given. It is an easy check that the criterion we give here is equivalent to the one in \cite{BrandenSolus}.

Let $ V := \langle t_k^{1/N_k},\, | \, k = 1, 2,\ldots,n+1\rangle_\CC$ be the vector subspace of $R$ spanned by the indeterminates $t_k^{1/N_k}$, $k = 1, \ldots, n+1$. Let $\zeta\in\CC$ be a $N$--th primitive root of unity and, for $j = 1,\ldots,n+1$, let $\gamma_j$ be the automorphism of $V$ defined by 
\[
\gamma_j\big(t_k^{1/N_k}\big) = \begin{cases} \zeta^{\,Nb_j/N_k}\, t_k^{1/N_k} & \textrm{if }k\leq j,\\[0.5cm] t_k^{1/N_k} & \textrm{if }k>j.\end{cases}
\]
We can extend $\gamma_j$ to an automorphism of $R$ by multiplicativity and $\CC$--linearity. Finally, let $\Gamma := \langle \gamma_1,\gamma_2,\dots,\gamma_{n+1}\rangle \subset \GL( V)\subset \mathrm{Aut}_\CC( R)$. This is a finite abelian group.

\begin{theorem}\label{thm:invariant-ring-Gamma}
The algebra $A(b)$ is the subalgebra of $\Gamma$--invariants in $ R$, i.e. $A(b) =  R^{\Gamma}$. Moreover, the subgroup $\Gamma \subset \GL( V)$ is free of pseudo-reflections.
\end{theorem}
\begin{proof}
The action of $\Gamma$ is diagonal on the monomial basis of $R$, so a general element of $R$ is $\Gamma$--invariant if and only if each monomial appearing in it is $\Gamma$--invariant. Hence it is enough to check which monomials $t^\pi = t_1^{a_1}\cdots t_{n+1}^{a_{n+1}}$, $\pi = (a_1, \ldots,a_{n+1})$, are $\Gamma$--invariant.
We have, $t_k^{a_k} = (t_k^{1/N_k})^{m_k}$ with $m_k := N_k a_k \in \ZZ_{\ge 0}$, so
\[
\gamma_j(t^\pi) = \gamma_j((t_1^{1/N_1})^{m_1})\cdots \gamma_j((t_{n+1}^{1/N_{n+1}})^{m_{n+1}}) = \zeta^{b_j(m_1N/N_1 + \cdots + m_jN/N_j)}\, t^\pi.
\]
Hence $t^\pi$ is $\Gamma$--invariant if and only if
\[
b_j(m_1N/N_1 + \cdots + m_jN/N_j)
\]
is a multiple of $N$ for all $j = 1, \ldots, n+1$. But $m_kN/N_k = N a_k$ for all $k$, so the previous condition is equivalent to $b_j(a_1 + \cdots + a_j) \in \ZZ$ for all $j = 1, \ldots, n+1$. By Proposition~\ref{prop:LS-characterization}, this is exactly the condition for $\pi$ to be an L-S path. This completes the proof that $A(b) =  R^{\Gamma}$.

Now we show that $\Gamma$ is free of pseudo-reflections. For an arbitrary element $\gamma = \gamma_1^{\ell_1}\cdots \gamma_{n+1}^{\ell_{n+1}}$ of $\Gamma$
\[
\gamma(t_k^{1/N_k}) = \zeta^{N c_k /N_k}\, t_k^{1/N_k},\quad c_k := b_k \ell_k + \cdots + b_{n+1} \ell_{n+1}.
\]
So suppose that $\gamma$ has all eigenvalues $\lambda_1, \ldots, \lambda_{n+1}$ equal to $1$ but at most one of them; i.e. there exists $1\leq j\leq n+1$ such that $\lambda_k = 1$ for all $k\neq j$. We will show that $\gamma$ is the identity.

The condition $\lambda_k = 1$ is equivalent to $N_k\,|\,c_k$ and, since $N_k = \mathrm{lcm}(b_k,b_{k-1})$, this is equivalent to $b_k\,|\,c_k$ and $b_{k - 1}\,|\,c_k$. Furthermore, since $c_k = b_k\ell_k + \cdots + b_{n+1}\ell_{n+1}$, the condition $b_k\,|\,c_k$ is equivalent to $b_k\,|\,c_{k+1}$ (where we put $c_{n+2} := 0$).

So, denoting by ${\mathsf C}_k$ the condition $b_k\,|\,c_{k+1}$, we have proved that $\lambda_k = 1$ if and only if ${\mathsf C}_{k-1}$ and ${\mathsf C}_k$ hold. The condition ${\mathsf C}_{n+1}$ is always true since $b_{n+1} = 1$. So, our hypothesis on $\gamma$ implies that ${\mathsf C}_k$ and ${\mathsf C}_{k-1}$ hold for all $k\neq j$. Hence ${\mathsf C}_k$ holds for each $k=1,\ldots,n+1$, and $\gamma$ is the identity.
\end{proof}

Using a standard result of invariant theory, the previous theorem gives a neat criterion for $A$ to be Gorenstein.
\begin{theorem}\label{thm:Gorenstein}
The algebra $A(b)$ is Gorenstein if and only if
\[
\displaystyle(\frac{1}{N_1},\ldots,\frac{1}{N_{n+1}})
\]
is an L-S path, i.e. if and only if
\[
b_j\cdot\left(\frac{1}{N_1} + \cdots + \frac{1}{N_j}\right) \in \NN, \quad \textrm{for each } j = 1, \ldots, n + 1.
\]
\end{theorem}
\begin{proof}
    By Theorem~\ref{thm:invariant-ring-Gamma}, $A(b)$ is the invariant subalgebra of a finite group $\Gamma$ of automorphisms of $ R$ free of pseudo-reflections. By \cite{WatanabeI, WatanabeII} (see also \cite{BrunsHerzog}) $A(b)$ is Gorenstein if and only if $\Gamma \subset \SL( V)$. This condition is clearly equivalent to $\det(\gamma_j) = 1$ for all $j = 1, \ldots, n+1$, which, in turn, is equivalent to
    \[
    b_j\cdot\left(\frac{1}{N_1} + \cdots + \frac{1}{N_j}\right) \in \NN, \quad j = 1, \ldots, n+1
    \]
    by computing the determinant of $\gamma_j$, $j = 1,\ldots, n+1$.
\end{proof}

\begin{corollary}\label{cor:Gorenstein}
    For each fixed $n\geq 1$, there are only finitely many bonds $b = (b_1, \ldots, b_n)$ such that the algebra $A(b)$ is Gorenstein.
\end{corollary}
\begin{proof}
    By Theorem~\ref{thm:Gorenstein}, if $A(b)$ is Gorenstein then $d := \frac{1}{N_1} + \cdots + \frac{1}{N_{n+1}}$ is a positive integer. Moreover, since $N_k \geq b_k \geq 1$ for all $k = 1, \ldots, n+1$, we have $d \leq n + 1$. The conclusion follows from the slight more general fact which is proved below.

    Consider the following claim: let $m$ be a positive integer and let $X$ be a finite subset of $\R$; then the set of $(x_1, \ldots, x_m)\in\NN^m$ such that $\frac{1}{x_1} + \cdots + \frac{1}{x_m}\,\in\,X$ is finite.

    We prove this claim by induction on $m$, the base case $m=1$ being trivial. Assume that the claim holds for $m - 1$ and let us prove it for $m$. Let $x_1, \ldots, x_m$ be positive integers such that $\frac{1}{x_1} + \cdots + \frac{1}{x_m}\,\in\,X$.

    If $X\cap\RR_{>0} = \varnothing$, then there is no solution. Assume otherwise and let $M$ be the minimum of $X\cap \R_{>0}$. If $x_j > m/M$ for all $j = 1, \ldots, m$, then $\frac{1}{x_1} + \cdots + \frac{1}{x_m} < M$, a contradiction. Hence, one of the indeterminates, say $x_m$, is bounded from above by $m/M$. So $(x_1, \ldots, x_{m-1})$ is a solution of $\frac{1}{x_1} + \cdots + \frac{1}{x_{m-1}} \in X'$ with $X' := \{x - \frac{1}{a}\,|\, x\in X, a = 1, \ldots, \lfloor m/M\rfloor\}$, which is a finite set. By induction, there are only finitely many solutions $(x_1, \ldots, x_{m-1})$ for $X'$, hence only finitely many solutions $(x_1, \ldots, x_m)$ for $X$.
\end{proof}

Denote by $\LS^\circ(b)$ the set of L-S paths $\pi = (a_1, \ldots, a_{n+1})$ such that $a_j > 0$ for all $j = 1, \ldots, n+1$; it is clear that $\LS^\circ(b) = \LS(b)\cap\RR^{n+1}_{>0}$, i.e. it is the set of L-S paths in the interior of the positive orthant. 

\begin{corollary}\label{cor:Gorenstein-open-part}
If $A(b)$ is Gorenstein then $\LS^\circ(b) = \pi^0 + \LS(b)$, where
\[
\pi^0 = (1/N_1, \ldots, 1/N_{n+1})
\]
is the unique L-S path of minimal degree in $\LS^\circ(b)$.
\end{corollary}
 \begin{proof} Since $A(b)$ is Gorenstein, by Theorem~\ref{thm:Gorenstein} we have $\pi^0\in\LS(b)$; hence $\pi^0 + \LS(b)\subseteq\LS^\circ(b)$. Conversely, let $\pi = (a_1, \ldots, a_{n+1})\in\LS^\circ(b)$; then $a_j > 0$ for all $j = 1, \ldots, n+1$. By Proposition~\ref{lem:coordinates-in-frac}, $a_j\geq 1/N_j$ for all $j = 1, \ldots, n+1$. Hence $\pi - \pi^0\in\RR^{n+1}_{\ge0}$ and we have $\pi - \pi^0\in\LS(b)$ by Proposition~\ref{prop:ls-normality}(ii). So $\pi \in\pi^0 +\LS(b)$.
\end{proof}

\smallskip

The lecture hall polytope $P(b)$ is said to be \emph{Gorenstein} if its Ehrhart ring $E(b)$ is Gorenstein. By Proposition~\ref{prop:LH-Ehrhart-ring}, this is equivalent to the algebra $A(b)$ being Gorenstein. Hence the following theorem is a direct consequence of Theorem~\ref{thm:Gorenstein} and of Corollary~\ref{cor:Gorenstein}.
\begin{theorem}\label{thm:Gorenstein-LH}
Let $b=(b_1,\ldots,b_n)$ be a set of bonds. The lecture hall polytope $P(b)$ is Gorenstein if and only if
\[
b_j\cdot\left(\frac{1}{N_1} + \cdots + \frac{1}{N_j}\right) \in \ZZ, \quad \textrm{for each }j = 1, \ldots, n+1.
\]
Moreover, for each fixed $n\geq 1$, there are only finitely many bonds $b = (b_1, \ldots, b_n)$ such that the lecture hall polytope $P(b)$ is Gorenstein.
\end{theorem}

%%%%%%%%%%%%%%%%%%%%%%%%%%%%%%%%%%%%%%%%%%%%%%%%%%%%%%%%%%%%%%%%%%
%%%%%%%%%%%%%%%%%%%%%%%%%%%%%%%%%%%%%%%%%%%%%%%%%%%%%%%%%%%%%%%%%%
%%%%%%%%%%%%%%%%%%%%%%%%%%%%%%%%%%%%%%%%%%%%%%%%%%%%%%%%%%%%%%%%%%
%%%%%%%%%%%%%%%%%%%%%%%%%%%%%%%%%%%%%%%%%%%%%%%%%%%%%%%%%%%%%%%%%%
%%%%%%%%%%%%%%%%%%%%%%%%%%%%%%%%%%%%%%%%%%%%%%%%%%%%%%%%%%%%%%%%%%

\section{The Koszul property}\label{Sec:6}

For a certain class of lecture hall polytopes, the Koszul property has been already proved in \cite{BrandenSolus}. We present here a general proof of this property for all lecture hall polytopes. We will show that the algebra $A(b)$ is Koszul by proving that it is the quotient of a polynomial ring by an ideal having a quadratic Gr\"obner basis. We will develop some standard monomial theory; mostly in an ad hoc way for our purposes.

Consider the polynomial ring $\mcA(b) := \CC[x_\pi\,|\,\pi\in\LS(b)_1]$ and the homomorphism of graded algebras
\[
\phi: \mcA(b) \longrightarrow A(b), \qquad x_\pi \longmapsto t^\pi.
\]
The kernel of $\phi$ is a homogeneous ideal $\mcI$ of $\mcA(b)$ and, being $A(b)$ generated in degree $1$ by Proposition \ref{prop:A-algebra}, $\phi$ is surjective and $A(b) = \mcA(b) / \mcI$.

In Section \ref{section:semigroup-algebra} we have introduced the torus $T=\Spec\CC[L(b)]$ and its natural action on $A(b)$. Here, we define an action of the torus $T$ on $\mcA(b)$ by letting $T$ act on the indeterminates $x_\pi$, $\pi\in\LS(b)_1$, via the character $\chi^\pi$. Then $\phi$ is a $T$--equivariant homomorphism of graded algebras; in particular the ideal $\mcI$ is $T$--homogeneous.

We say that a monomial $x_{\pi_1}\cdots x_{\pi_k}$ of $\mcA(b)$ is in \emph{standard form} if  $\max\supp\pi_j\leq \min\supp\pi_{j+1}$ for all $j=1,\ldots,k-1$. A monomial $x_{\pi_1}\cdots x_{\pi_k}$ is \emph{standard} if it can be rearranged into standard form; otherwise it is \emph{non-standard}. The following lemma is a direct consequence of the definitions.
\begin{lemma}\label{lem:standard-form-canonical-decomposition}
    The monomial $x_{\pi_1}\cdots x_{\pi_k}$ is in standard form if and only if $\pi_1 \oplus \cdots \oplus \pi_k$ is the canonical decomposition of the L-S path $\pi_1 + \cdots + \pi_k$.
\end{lemma}

Now we define a monomial order in $\mcA(b)$. First of all, consider the following order on the L-S paths of degree 1: let $\pi = (a_1,\ldots,a_{n+1})$, $\eta = (c_1, \ldots, c_{n+1})\in\LS(b)_1$, then $\pi \succ \eta$ if and only if the leftmost non-zero entry of $(a_1 - c_1, a_2 - c_2, \ldots, a_{n+1} - c_{n+1})$ is negative. We accordingly order the indeterminates: $x_\pi \succ x_\eta$ if and only if $\pi \succ\eta$.

We denote by the same symbol $\succ$ the graded reverse lexicographic order on monomials in $\mcA(b)$ associated to this ordering of the indeterminates $x_\pi$, $\pi\in\LS(b)_1$. In other words, if $\LS(b)_1 = \{\eta_1 \succ \cdots \succ \eta_s\}$ and $M = x_{\eta_1}^{\alpha_1}\cdots x_{\eta_s}^{\alpha_s}$, $M' = x_{\eta_1}^{\beta_1}\cdots x_{\eta_s}^{\beta_s}$ are two monomials of $\mcA(b)$, then $M \succ M'$ if and only if $\deg(M) > \deg(M')$ or $\deg(M) = \deg(M')$ and the rightmost non-zero entry of $(\alpha_1 - \beta_1, \ldots, \alpha_s - \beta_s)$ is negative.

In the following, in order to align with the definition of canonical decomposition for L-S paths, we will sometimes write monomials in ascending order of the indeterminates with respect to $\succ$. Indeed we have the following lemma.

\begin{lemma}\label{lem:standard-form}
    If the monomial $x_{\pi_1}\cdots x_{\pi_k}$ is in standard form, then $x_{\pi_1}\preceq\cdots\preceq x_{\pi_k}$.
\end{lemma}
\begin{proof}
We use induction on $k$. The case $k=1$ is trivial. Assume $k>1$ and note that $x_{\pi_1}\cdots x_{\pi_{k-1}}$ is in standard form, hence by induction $x_{\pi_1}\preceq\cdots\preceq x_{\pi_{k-1}}$. So we just need to prove that $x_{\pi_{k-1}}\preceq x_{\pi_k}$; that is, the case $k=2$. So for notation simplicity we can assume $k=2$.

By the definition of standard form we have $\max\supp\pi_1\leq \min\supp\pi_2$. Let $j := \min\supp\pi_1$; note that $j \leq \max\supp\pi_1\leq \min\supp\pi_2$. If $j < \min\supp\pi_2$ then $\pi_1$ has a non-zero coordinate at position $j$ whereas $\pi_2$ has each coordinates at positions $1, 2, \ldots, j$ equal to $0$. Hence $x_{\pi_1}\prec x_{\pi_2}$.

On the other hand, if $j = \min\supp\pi_2$ then $j = \min\supp\pi_1 = \max\supp\pi_1$, hence $\pi_1 = e_j$. Now, if $\pi_2 = e_j$ then the claim is true, otherwise the $j$--th coordinate of $\pi_2$ is strictly less than $1$ and $x_{\pi_1}\preceq x_{\pi_2}$.
\end{proof}

\begin{lemma}\label{lem:standard}
Let M = $x_{\pi_1}\cdots x_{\pi_k}$ be a monomial in $\mcA(b)$ with $x_{\pi_1} \preceq \cdots \preceq x_{\pi_k}$. The followings are equivalent:
\begin{itemize}
    \item[(i)] $M$ is standard,
    \item[(ii)] $M$ is in standard form, 
    \item[(iii)] $x_{\pi_j}x_{\pi_{j+1}}$ is standard for all $j=1,\ldots,k-1$.
\end{itemize}
\end{lemma}
\begin{proof}
It is clear that \textrm{(ii)} implies \textrm{(i)}. For the converse, suppose that $M$ is standard. Then a rearrangement of $M$ gives a standard form; but the previous lemma implies that no rearrangment is needed, hence $M$ is in standard form.

Being in standard form is clearly a condition about pairs of consecutive factors, hence \textrm{(ii)} and \textrm{(iii)} are equivalent by \textrm{(i)}.
\end{proof}

\begin{lemma}\label{lem:standard-monomial-order}
Let $\pi\in\LS(b)_k$. In the set of monomials $x_{\pi_1}\cdots x_{\pi_k}$ of $\mcA(b)$ such that $\phi(x_{\pi_1}\cdots x_{\pi_k}) = t^\pi$, there exists a unique standard monomial; the standard form of this monomial is $x_{\eta_1}\cdots x_{\eta_k}$ where $\eta_1 \oplus \cdots \oplus \eta_k$ is the canonical decomposition of $\pi$. Moreover, $x_{\eta_1}\cdots x_{\eta_k}$ is the minimum of this set with respect to $\succ$.
\end{lemma}
\begin{proof} It is clear that $\phi(x_{\eta_1}\cdots x_{\eta_k}) = t^\pi$, moreover $x_{\eta_1}\cdots x_{\eta_k}$ is in standard form by Lemma \ref{lem:standard-form-canonical-decomposition}. Now let $x_{\pi_1}\cdots x_{\pi_k}$ be another monomial such that $\phi(x_{\pi_1}\cdots x_{\pi_k}) = t^\pi$. Since the canonical form of $\pi$ is unique, if $x_{\pi_1}\cdots x_{\pi_k}$ is standard then it must be equal to $x_{\eta_1}\cdots x_{\eta_k}$. On the other hand, suppose that $x_{\pi_1}\cdots x_{\pi_k}$ is non-standard; we want to prove that $x_{\pi_1}\cdots x_{\pi_k} \succ x_{\eta_1}\cdots x_{\eta_k}$.

We use induction on the degree $k$ of the monomial. The case $k=1$ is trivial, so we can assume that $k>1$. By the definition of the canonical decomposition we have $x_{\eta_1}\preceq x_{\eta_2}\preceq \cdots \preceq x_{\eta_k}$. We can also assume that $x_{\pi_1}\preceq x_{\pi_2}\preceq \cdots \preceq x_{\pi_k}$ (since $\mcA(b)$ is commutative).

If $x_{\pi_1} = x_{\eta_1}$ then we can apply the induction hypothesis to the monomials $x_{\pi_2}\cdots x_{\pi_k}$ and $x_{\eta_2}\cdots x_{\eta_k}$, which have degree $k-1$, and we are done. So we can assume that $x_{\pi_1}\neq x_{\eta_1}$ and we have to prove that $x_{\pi_1} \succ x_{\eta_1}$.

Let $\eta_1 = (a_1, \ldots, a_{n+1})$, $\pi_1 = (a_1',\ldots,a'_{n+1})$ and $\pi = (c_1,\ldots,c_{n+1})$. Let $j$ be the minimum index such that $a_j \neq a'_j$, we prove that $a_j > a'_j$. By the definition of the canonical decomposition and of $j$ we have $a_1 = a'_1 = c_1$, $a_2 = a'_2 = c_2$, $\ldots$, $a_{j-1} = a'_{j-1} = c_{j-1}$.

If $c_j = a_j$ then $a_j > a'_j$ since $\pi_1 + \cdots + \pi_k = \pi = \eta_1 + \cdots + \eta_k$. On the other hand, if $c_j > a_j$ then, by the definition of the canonical decomposition, $a_{j+1} = \cdots = a_{n+1} = 0$. So $a_1 + \cdots + a_j = 1 = \deg\eta_1$ and $a'_1 + \cdots + a'_j \leq 1 = \deg\pi_1$; hence, using $a_h = a'_h$ for $h = 1,\ldots, j - 1$, we have $a_j \geq a'_j$. But if $a_j = a'_j$ then $\pi_1 = \eta_1$ which is impossible since we are assuming $x_{\pi_1}\neq x_{\eta_1}$. So, in both cases we have proved that $a_j > a'_j$, which implies that $x_{\pi_1} \succ x_{\eta_1}$, as desired.
\end{proof}

\begin{proposition}\label{prop:standard-monomials}
    The image under $\phi$ of the set of standard monomials of $\mcA(b)$ is a basis of $A(b)$ as a vector space. For each non-standard monomial $x_{\pi_1}\cdots x_{\pi_k}$ of $\mcA(b)$ there exists a relation, called \emph{straightening relation}, of the form
    \[
    x_{\pi_1}\cdots x_{\pi_k} - x_{\eta_1}\cdots x_{\eta_k}\,\,\in\,\,\mcI
    \]
    where $\eta_1 \oplus \cdots \oplus \eta_k$ is the canonical decomposition of the L-S path $\pi_1 + \cdots + \pi_k$; in particular $x_{\eta_1}\cdots x_{\eta_k}$ is a standard monomial. Moreover the quadratic straightening relations generate the ideal $\mcI$.
\end{proposition}
\begin{proof}
The set $t^\pi$, $\pi\in\LS(b)$, is a basis of $A(b)$ as a vector space. By Lemma \ref{lem:standard-monomial-order}, there is a unique standard monomial $x_{\pi_1}\cdots x_{\pi_k}$, $k = \deg\pi$, such that $\phi(x_{\pi_1}\cdots x_{\pi_k}) = t^\pi$. This proves the first claim.

Now suppose that $x_{\pi_1}\cdots x_{\pi_k}$ is a non-standard monomial of $\mcA(b)$ and let $\eta_1 \oplus \cdots \oplus \eta_k$ be the canonical decomposition of the L-S path $\pi = \pi_1 + \cdots + \pi_k$. Then $x_{\eta_1}\cdots x_{\eta_k}$ is a standard monomial and, by definition of $\phi$, we have $\phi(x_{\pi_1}\cdots x_{\pi_k}) = t^\pi = \phi(x_{\eta_1}\cdots x_{\eta_k})$. Hence $x_{\pi_1}\cdots x_{\pi_k} - x_{\eta_1}\cdots x_{\eta_k}\in\mcI$.

It remains to prove that the quadratic straightening relations generate the ideal $\mcI$. First we show that $\mcI$ is generated by all straightening relations. Being $\mcI$ $T$--homogeneous, it is enough to show that each $T$--homogeneous element of $\mcI$ is a linear combination of straightening relations.

Let $\pi\in\LS(b)$, let $M_0, M_1, \ldots, M_h$ be the monomials in $x_{\eta}$, $\eta\in\LS(b)_1$, such that $\phi(M_i) = t^\pi$, $i=0,\ldots,h$, let $M_0$ be the unique standard monomial in this set and let $f := \sum_{i=0}^h c_i M_i \in \mcI$. Since $\phi(f) = \left(\sum_{i = 0}^h c_i\right)t^\pi$, we find $\sum_{i = 0}^h c_i = 0$. Hence $f = \sum_{i=0}^h c_i (M_i - M_0)$, a linear combination of straightening relations.

Finally we show that the quadratic straightening relations generate all the straightening relations. Let $x_{\pi_1}\cdots x_{\pi_k}$ be a non-standard monomial. By Lemma \ref{lem:standard}(ii), there exists $j$ such that $x_{\pi_j}x_{\pi_{j+1}}$ is non-standard. Hence, we can apply a quadratic straightening relation and write $x_{\pi_j}x_{\pi_{j+1}} = x_{\eta_j}x_{\eta_{j+1}}$ up to an element of $\mcI$. By Lemma \ref{lem:standard-monomial-order}, $x_{\pi_j}x_{\pi_{j+1}} \succ x_{\eta_j}x_{\eta_{j+1}}$.

We can now repeat this process with the monomial $x_{\pi_1}\cdots x_{\pi_{j-1}}x_{\eta_j}x_{\eta_{j+1}}x_{\pi_{j+2}}\cdots x_{\pi_k}$. Since the set of monomials $x_{\pi_1}\cdots x_{\pi_k}$ such that $\phi(x_{\pi_1}\cdots x_{\pi_k}) = t^\pi$ is finite, this process must terminate. This proves that the quadratic straightening relations generate all the straightening relations, and the proof is complete.
\end{proof}

\begin{theorem}\label{thm:quadratic-Groebner}
    The set of quadratic straightening relations of Proposition \ref{prop:standard-monomials} is a Gr\"obner basis of the ideal $\mcI$ with respect to the monomial order $\succ$. In particular, $A(b)$ is Koszul.
\end{theorem}
\begin{proof}
    For a non-standard monomial $x_{\pi_1}x_{\pi_2}$ let $y_{\pi_1,\pi_2} = x_{\pi_1}x_{\pi_2} - x_{\eta_1}x_{\eta_2}$ be the corresponding quadratic straightening relation, where $x_{\eta_1}x_{\eta_2}$ is the standard monomial with the same image under $\phi$. Then the leading term of $y_{\pi_1,\pi_2}$ with respect to the monomial order is $x_{\pi_1}x_{\pi_2}$, by Lemma \ref{lem:standard-monomial-order}. Hence the ideal generated by the leading terms of the quadratic straightening relations is exactly the vector space $\mathcal{V}$ spanned by the non-standard monomials.
    
    We want to show that $\mathcal{V}$ is the ideal generated by the leading terms of all the elements of $\mcI$. By contradiction, suppose otherwise and let $f = M_0 + f_1\in\mcI$ with $\initial(f) = M_0$ a standard monomial and $f_1$ a linear combination of monomials strictly smaller than $M_0$. 
    
    Using the straightening relations, we can replace any non-standard monomial $M$ in $f_1$ by a  standard monomial $M'$. Note that, by Lemma \ref{lem:standard-monomial-order}, $M' \prec M$. Hence these replacements cannot delete the leading term $M_0$ of $f$.

    So we get a new element $f' = M_0 + f'_1\in \mcI$ that is a non-zero linear combination of standard monomials. This is impossible since the standard monomials are linearly independent in $A(b)$.

    Finally, since $\mcI$ has a quadratic Gr\"obner basis, $A(b)$ is Koszul by \cite{Froberg} and \cite{Anick}.
\end{proof}

It is now sufficient to translate the Koszul property of $A(b)$ to the Ehrhart ring $E(b)$ of the lecture hall polytope $P(b)$.
\begin{theorem}\label{thm:Koszul-LH}
Let $b=(b_r,\ldots,b_1)$ be a set of bonds. Then the Ehrhart ring $E(b)$ of the lecture hall polytope $P(b)$ is Koszul.
\end{theorem}
\begin{proof}
By Proposition~\ref{prop:LH-Ehrhart-ring}, the algebra $E(b)$ is isomorphic to the algebra $A(b)$. By Theorem~\ref{thm:quadratic-Groebner}, the algebra $A(b)$ is Koszul, hence $E(b)$ is Koszul.
\end{proof} 

%%%%%%%%%%%%%%%%%%%%%%%%%%%%%%%%%%%%%%%%%%%%%%%%%%%%%%%%%%%%%%%%%%
%%%%%%%%%%%%%%%%%%%%%%%%%%%%%%%%%%%%%%%%%%%%%%%%%%%%%%%%%%%%%%%%%%
%%%%%%%%%%%%%%%%%%%%%%%%%%%%%%%%%%%%%%%%%%%%%%%%%%%%%%%%%%%%%%%%%%
%%%%%%%%%%%%%%%%%%%%%%%%%%%%%%%%%%%%%%%%%%%%%%%%%%%%%%%%%%%%%%%%%%
%%%%%%%%%%%%%%%%%%%%%%%%%%%%%%%%%%%%%%%%%%%%%%%%%%%%%%%%%%%%%%%%%%
\section{Schubert varieties, degenerations and LH-toric varieties}\label{Sec:7}

This is just a quick summary, to simplify the notation set $\mathbb K=\mathbb C$. 
Let $G$ be a complex semi-simple simply connected algebraic group, or 
a symmetrizable Kac-Moody group. Let $B$ be a Borel subgroup and let $Q\supseteq B$ be a parabolic subgroup. 
Denote by $W$ and $W_Q$ the respective Weyl groups of $G$ and $Q$. For $w\in W/W_Q$ let $X_w\subseteq G/Q$ be the 
associated Schubert variety. We say a dominant weight $\lambda$ is $Q$-regular if $Q$ is the normalizer in $G$ of the line $\mathbb C v_\lambda$ through a highest 
weight vector in the irreducible $G$-representation $V(\lambda)$.
For a $Q$-regular dominant weight $\lambda$ and $w\in W/W_Q$ let $V(\lambda)_w$ be the corresponding Demazure module.
It is finite dimensional, even in the symmetrizable Kac-Moody case, and one has an embedding $X_w\subseteq \mathbb P(V(\lambda)_w)$.

Using the standard monomial theory developed in \cite{LittelmannPaths3, LLM} (see \cite{CFL1, CFL4} for a modern presentation), 
it was shown in \cite{Chirivi} that $X_w$ admits a flat semi-toric degeneration, i.e.
there exists a variety $Y$ together with a faithfully flat morphism $\pi:Y\rightarrow \mathbb A^1$, such that the general
fibre $\pi^{-1}(t)$ is isomorphic to $X_w$ for $t\not =0$, and the special fibre $\pi^{-1}(0)$ is a reduced union
of projective toric varieties. More precisely:
\begin{enumerate}
\item the number of irreducible components of the special fibre is equal to the number of maximal chains between $w$ and $\mathrm{id}$ in $W/W_Q$;
\item to a maximal chain $\mathbf w: w=w_0>w_{1}>\ldots>w_n=\mathrm{id}$ one can associate a  sequence of bonds $b_\mathbf w=(b_1,\ldots,b_n)$
as follows: for $j=1,\ldots, n$ let $\beta_j$ be a positive root such that $w_{j-1}=s_{\beta_j}w_{j}$. We set $b_j=\langle w_{j}(\lambda),\beta^\vee_j\rangle$,
$j=1,\ldots,n$;
\item 
let $A(b_\mathbf w)$ be semigroup algebra defined in Section~\ref{section:semigroup-algebra}. The irreducible
component of $\pi^{-1}(0)$ associated to $\mathbf w$ is isomorphic to
${\rm Proj\,}A(b_\mathbf w)$, called the \emph{lecture hall toric variety} (for short LH-toric variety) associated to $b_\mathbf w$.
\end{enumerate}

\begin{example}\rm
Let $b=(b_1,\ldots,b_n)$ be a sequence of bonds. To find an embedded Schubert variety $X_w\subseteq \mathbb P(V(\lambda)_w)$ 
such that $b=b_\mathbf w$ for some maximal chain $\mathbf w$, set $G=\mathrm{SL}_{2n}(\mathbb K)$. 
Let $Q$ be the parabolic subgroup containing all the negative simple roots $(-\alpha_{2i})$, $i=1,\ldots,n-1$, set 
$\lambda=\sum_{i=1}^n b_i\omega_{2i-1}$, $w=s_{\alpha_{2n-1}}\cdots s_{\alpha_5}s_{\alpha_3}s_{\alpha_1}$
and set $\mathbf w=(w,\ldots,s_{\alpha_3}s_{\alpha_1}, s_{\alpha_1},\mathrm{id})$. Using the rules in (2) above one finds: $b_\mathbf w=(b_1,\ldots,b_n)$.

The simple roots used in the definition of $w$ commute. So one finds for any 
other maximal chain $\mathbf w'$: $b_{\mathbf w'}=(b_{\sigma(1)},\ldots,b_{\sigma(n)})$ for some permutation $\sigma\in S_n$,
and all permuations occur. Indeed, $X_w$ is in this case 
a toric variety, it is a product of $\mathbb P^1$'s. The Newton-Okounkov body of $X_w\subseteq \mathbb P(V(\lambda)_w)$ 
(with respect to the flag of Schubert varieties corresponding to the elements in $\mathbf w$)
is a rectangular box, and the flat semi-toric degeneration is triangulating this box. 

These kind of triangulations of Newton-Okounkov bodies using quasi-valuations have been studied in \cite[Section 7]{FL} in the case of Gra\ss mannians. The degeneration of the toric varieties associated to the Newton-Okounkov bodies into the semi-toric varieties of the triangulations can be done using weight matrices in the higher rank Gelfand-Kapranov-Zelevinsky cones \cite{CCFL}.
\end{example}
\begin{example}\rm
Another interesting case is the Gra\ss mann variety for the affine Kac-Moody groups $\widehat{\mathrm{SL}}_2(\mathbb K)$. We have two simple roots 
$\alpha_0,\alpha_1$. Let $Q$ be the maximal parabolic subgroup containing $-\alpha_1$,
and take $\lambda=\omega_0$. In this case one has for every $n\in \mathbb N$ a unique Schubert variety $X_{w_n}$
of dimension $n$, where 
$$
w_n=\underbrace{\cdots s_{\alpha_0}s_{\alpha_1}s_{\alpha_0}}_{n\ {\rm factors}} \in W/W_Q.
$$
There exists only
one maximal chain $\mathbf w_n$ between $\mathrm{id}$ and $w_n$, and the corresponding set of bonds is: 
$b_{\mathbf w_n}=(1,2,3,\ldots,n)$. So we recover the original lecture hall polytopes via the Schubert varieties in the affine Gra\ss mannian for $\widehat{\mathrm{SL}}_2(\mathbb C)$.
\end{example}

\end{document}